 \documentclass[12pt]{amsart}
\usepackage[top=1in, bottom=1in, left=1in, right=1in]{geometry}

\usepackage{amssymb,mathrsfs,amsmath}

\newtheorem{theorem}{Theorem}[section]
\newtheorem{proposition}{Proposition}[section]
\newtheorem{lemma}[theorem]{Lemma}
\newtheorem{corollary}[theorem]{Corollary}
\theoremstyle{remark}

\newcommand{\bsp}{\begin{split}}
\newcommand{\esp}{\end{split}}
\newcommand{\be}{\begin{equation}}
\newcommand{\ee}{\end{equation}}
\newcommand{\bes}{\begin{equation*}}
\newcommand{\ees}{\end{equation*}}
\newcommand{\bv}\boldsymbol{}

\numberwithin{equation}{section}

\renewcommand{\pmod}[1]{~({\rm mod}\,#1)}

\begin{document}

\title{Mean values of multiplicative functions over function fields} 
\author{Andrew Granville}
\address{Centre de recherches math\'ematiques\\
Universit\'e de Montr\'eal\\
CP 6128 succ. Centre-Ville\\
Montr\'eal, QC H3C 3J7\\
Canada}
\email{{\tt andrew@dms.umontreal.ca}} 
\author{Adam J Harper}
\address{Jesus College \\ Cambridge \\ CB5 8BL \\ England} 
\email{\tt A.J.Harper@dpmms.cam.ac.uk}
\author{Kannan Soundararajan}
\address{Department of Mathematics\\ Stanford University\\ Stanford, CA 94305\\ USA} 
\email{\tt ksound@stanford.edu} 
\abstract{We discuss the mean values of multiplicative functions over function fields.  In particular, we adapt the authors' new proof of Hal{\' a}sz's theorem on mean values to this simpler setting.  Several of the technical difficulties that arise over the integers disappear in the function field setting, which helps bring out more clearly the main ideas of the proofs over number fields.  
We also obtain Lipschitz estimates showing the slow variation of mean values of multiplicative functions over function fields, which display some features that are not present in the integer situation.}
\endabstract
\thanks{Andrew Granville's research is partly supported by NSERC, Canada, and a Canadian Research Chair. Adam Harper is supported by a research fellowship at Jesus College, Cambridge. Kannan Soundararajan was partially supported by NSF grant DMS 1001068, and a Simons Investigator grant from the Simons Foundation.}

\date{\today} 
\maketitle 

\section{Introduction} 

We begin by introducing multiplicative functions over the polynomial ring ${\Bbb F}_q[x]$, highlighting 
the analogy with multiplicative functions over the integers.  In the subsequent sections of the introduction 
we will discuss the new results in this paper.

\subsection{An introduction to multiplicative functions over function fields} 

In the polynomial ring ${\Bbb F}_q[x]$, where $q$ is a prime power, let ${\mathcal M}$ denote the set of 
monic polynomials and let ${\mathcal M}_n$ denote the set of monic polynomials of degree $n$, so that $|{\mathcal M}_n| = q^n$.  Upper case letters like $F$, $G$ shall denote monic polynomials.  Let ${\mathcal P}$ denote the set of irreducible monic polynomials, and ${\mathcal P}_n$ those of degree $n$, and we reserve the letter $P$ to denote irreducible monic polynomials.  
We denote the degree of a polynomial $F$ by $\text{deg}(F)$. 


We are interested in multiplicative functions $f: {\mathcal M} \to {\Bbb C}$; that is, functions $f$ satisfying $f(FG)=f(F)f(G)$ for all  coprime monic polynomials $F$ and $G$.   The analogous functions  over the integers, namely multiplicative functions 
$f: {\Bbb N} \to {\Bbb C}$ (that is, functions $f$ with $f(mn)=f(m)f(n)$ for all coprime integers $m$ and $n$), have been extensively investigated.  A useful tool in studying multiplicative functions over the integers is the Dirichlet series 
$$ 
F(s) = \sum_{n=1}^{\infty} \frac{f(n)}{n^s} = \prod_{p} \Big( 1+ \frac{f(p)}{p^s}  + \frac{f(p^2)}{p^{2s}} + \ldots \Big), 
$$ 
where the product is over all primes $p$, and one usually restricts attention to those multiplicative functions for which the series 
and product are absolutely convergent in Re$(s)>1$.  Correspondingly, to study multiplicative functions over function fields,  
we put 
\begin{equation} 
\label{3} 
{\mathcal F}(z) = \sum_{F \in {\mathcal M}} f(F) z^{\text{deg}(F)} = \prod_{P} \Big( 1+ f(P) z^{\text{deg}(P)} + f(P^2) z^{2\text{deg}(P)} +\ldots \Big), 
\end{equation} 
where we assume that the series and product converge absolutely in $|z| <1/q$.  

In Section 2 we give several examples of interesting multiplicative functions over ${\Bbb F}_q[x]$, and for a general introduction to number theory over function fields we refer to \cite{Rosen}.   For the moment, it may be helpful to consider the most basic example, the multiplicative function taking the value $1$ on all monic polynomials in ${\Bbb F}_q[x]$.   Here, in the domain $|z|<1/q$ we have 
$$
{\mathcal F}(z) = \sum_{n=0}^{\infty} |{\mathcal M}_n| z^n=(1-qz)^{-1}=\prod_P (1-z^{\text{deg}(P)})^{-1}, 
$$
which corresponds to the Riemann zeta-function
$$
\zeta(s)=\sum_{n=1}^{\infty} \frac{1}{n^s} = \prod_p \Big(1-\frac{1}{p^s}\Big)^{-1}.
$$ 
Taking logarithms above yields that
\[
\sum_{m=1}^{\infty} \frac{(qz)^m}m = \sum_P \sum_{k=1}^{\infty} \frac{z^{k\, \text{deg}(P)}}k 
= \sum_P  \sum_{k= 1}^{\infty} \Lambda(P^k) \cdot \frac{z^{\text{deg}(P^k)}}{\text{deg}(P^k)} =
\sum_{n=1}^{\infty} \frac{z^n}n \cdot \sum_{F\in {\mathcal M}_n } \Lambda(F),
\]
where $\Lambda(F)$, in analogy with the von Mangoldt function of prime number theory, is defined to be zero unless $F=P^k$ is the power of an irreducible in which case  $\Lambda(F)= \text{deg}(P)$.  
Equating coefficients implies that 
\begin{equation} 
\label{2} 
\sum_{F\in {\mathcal M}_n } \Lambda(F) = q^n,
\end{equation} 
and now M{\" o}bius inversion gives the the well-known ``prime number theorem for ${\Bbb F}_q[x]$''
\begin{equation*}
|{\mathcal P}_n| = \frac 1n \sum_{d|n} \mu(d) q^{n/d} =\frac{q^n}{n} + O\Big(\frac{q^{n/2}}{n}\Big). 
\end{equation*} 
The analogous relationship in the integers, namely that $\sum_{n\leq x} \Lambda(n)=x+O(x^{1/2+o(1)})$, is an open question, equivalent to the Riemann Hypothesis.

For a general multiplicative function $f$, take logarithms in \eqref{3}, and write 
\begin{equation} 
\label{4.1bis} 
\log {\mathcal F}(z) = \sum_{F\in \mathcal{M}} \frac{\Lambda_f(F)}{\text{deg}(F)} z^{\text{deg}(F)}, 
\end{equation} 
for certain coefficients $\Lambda_f(F)$ with $\Lambda_f(F) = 0$ unless $F$ is the power of an irreducible.  
Differentiating, we may equivalently write \eqref{4.1bis} as 
\begin{equation} 
\label{4}
z\frac{{\mathcal F}^{\prime}}{{\mathcal F}}(z) = \sum_{F \in \mathcal{M}} \Lambda_f(F) z^{\text{deg}(F)}.
\end{equation} 
For a given positive real number $\kappa$ we focus on the class of multiplicative functions ${\mathcal C}(\kappa)$ 
consisting of  those $f$ for which  $f(1)=1$ and 
\begin{equation} 
\label{5} 
|\Lambda_f(F)| \le \kappa \Lambda(F) 
\end{equation} 
for all $F$.  The hypotheses \eqref{4} and \eqref{5} ensure the absolute convergence of the series 
and product in \eqref{3} for $|z|<1/q$.  

In \cite{GHS} we studied the analogous class of multiplicative functions $f$ over the integers 
for which $|\Lambda_f(n)| \le \kappa \Lambda(n)$, where $-F'(s)/F(s)=: \sum_{n\geq 1} \Lambda_f(n)/n^s$. 
The bound on $|\Lambda_f(n)|$ guarantees that the Dirichlet series and Euler product defining $F(s)$
 are absolutely convergent for Re$(s)>1$.

Given a multiplicative function $f$ in ${\mathcal C}(\kappa)$ our aim is to understand (for $n\ge 0$) the mean value 
\begin{equation} 
\label{6}  
\sigma(n) = \sigma(n;f) := \frac{1}{q^n} \sum_{M \in {\mathcal M}_n} f(M)  
\end{equation} 
in terms of the corresponding averages of $f$ over prime powers 
\begin{equation} 
\label{7} 
\chi(n) =\chi(n;f) := \frac{1}{q^n} \sum_{F \in {\mathcal M}_n} \Lambda_f(F).
\end{equation} 
We have $\sigma(0)=1,\ \chi(0)=0$ and $\sigma(1)=\chi(1)$, and then we observe (this follows from \eqref{4}, and will be justified in Remark 2 in Section \ref{examplesremarks} below) that 
\begin{equation}   
\label{2.1} 
n\sigma(n) = \sum_{k=1}^{n} \chi(k) \sigma(n-k).
\end{equation}
With this notation, we may also write \eqref{3} and  \eqref{4.1bis} as 
\begin{equation} \label{PowerSeries}
\mathcal{F}(z/q) = \sum_{n=0}^{\infty} \sigma(n) z^n = \exp\Big(\sum_{k=1}^{\infty} \frac{\chi(k)}k z^k \Big).
\end{equation}

The convolution relation \eqref{2.1} is a little more involved in the integer situation.  The discrete relation \eqref{2.1} is replaced by the continuous integral equation $u\sigma(u) = \int_0^u \chi(t) \sigma(u-t) dt$, where $\chi(t) = \psi(y^t)^{-1} \sum_{n\le y^t} \Lambda_f(n)$ is an average of the multiplicative function $f$ evaluated at prime powers (here $y$ is a suitably large parameter), and then $\sigma(u)$ approximates (in many situations) the mean-value of the function $f$ evaluated over integers up to $y^u$.  Such integral equations were first considered by Wirsing, and are discussed further in \cite{GSSpectrum}. 


\subsection{Hal\'asz's Theorem over function fields} 

In \cite{GHS} we show, generalizing a little the pioneering work of Hal{\' a}sz, that if $x$ is large, and if  $|\Lambda_f(n)|\le \kappa \Lambda(n)$ for all $n$ then 
\begin{equation} \label{NoH1} 
\frac 1x \sum_{n\le x} f(n)  \ll_{\kappa} \frac{1}{\log x} \int_{1/\log x}^1 \Big( \max_{|t| \le (\log x)^{\kappa} } \Big| \frac{F(1+\sigma+it)}{1+\sigma+it}\Big| 
  \Big) \frac{d\sigma}{\sigma} + \frac{(\log \log x)^{\kappa}}{\log x}. 
\end{equation} 
If one inserts a trivial bound $|F(1+\sigma+it)| \ll_{\kappa} 1/\sigma^{\kappa}$ on the right then one recovers the trivial bound $\ll_{\kappa} (\log x)^{\kappa - 1}$ for the left hand side (up to constants), and inserting {\em any} non-trivial information about $F(1+\sigma+it)$ supplies a non-trivial bound for the left hand side. This lossless quality is the crucial feature of Hal\'{a}sz-type theorems.

The left-hand side in \eqref{NoH1} is independent of the values of $f(p^k)$ on the prime powers $p^k>x$. Hence if we define $\Lambda_{f^\perp}(p^k)=\Lambda_{f}(p^k)$ for all prime powers with $p^k\leq x$, and $\Lambda_{f^\perp}(p^k)=0$ otherwise, then $f^\perp(n)=f(n)$ for each $n\leq x$, and the Dirichlet series
$F^\perp(s):=\sum_{n\geq 1} f^\perp(n)/n^s$ has the finite Euler product $\prod_{p\leq x} (1+f^\perp(p)/p^s+f^\perp(p^2)/p^{2s}+\ldots)$ which is analytic for all $s$ with Re$(s)>0$.   One can replace $F$ in the upper bound in \eqref{NoH1} with $F^\perp$, which is sometimes convenient. 

We now describe the corresponding result in the function field setting.  Define the multiplicative function $f^\perp$ by setting $\Lambda_{f^\perp}(M)= \Lambda_f(M)$ if $\deg M<n$, and $\Lambda_{f^\perp}(M)=0$  if $\deg M\geq n$.
Hence $\chi^\perp(m)=\chi(m)$ and $\sigma^\perp(m)=\sigma(m)$ for all $m<n$, whereas 
  $\sigma^\perp(n)=\sigma(n)-\chi(n)/n$ in view of \eqref{2.1}. Now ${\mathcal F}^\perp(z)$ is entire, whereas ${\mathcal F}(z)$ might only be analytic in a disc. Paralleling \eqref{NoH1},  we establish the following result.

  \begin{theorem}\label{Halasz} Let $f$ be in the class ${\mathcal C}(\kappa)$ and let $\sigma(n)$ be defined as in \eqref{6}.  Then for all $n\ge 1$ we have 
\begin{equation} \label{NoH2}  
| \sigma^{\perp}(n) | \le \frac{\kappa^2}{n} 
\int_0^1   \Big(\max_{|z|=\sqrt{t}} |  {\mathcal F}^\perp(z/q) | \Big) \Big(\frac{1-t^{n-1}}{1-t}\Big) dt, 
\end{equation} 
and therefore 
$$
|\sigma(n)| \le |\sigma^{\perp}(n)| + \Big|\frac{\chi(n)}{n} \Big| \le \frac{\kappa^2}{n} 
\int_0^1  \Big( \max_{|z|=\sqrt{t}} |  {\mathcal F}^\perp(z/q)|\Big)  \Big(\frac{1-t^{n-1}}{1-t}\Big) dt + \frac{\kappa}{n}.
$$ 
\end{theorem} 

Our short proof of Theorem \ref{Halasz} is given in Section 3.   Note the strong parallel between this theorem and the corresponding estimate \eqref{NoH1}.   In our bound for $\sigma(n)$ above, the term $\kappa/n$ in the bound arose from the contribution of irreducibles of degree $n$.  Correspondingly, in \eqref{NoH1} the error term $(\log \log x)^{\kappa}/\log x$ includes the contribution from primes near $x$, but also includes contributions from certain other numbers, and from error terms in truncating Perron integrals, and  these do not arise in the simpler function field setting. 


As in \eqref{PowerSeries},  
 \begin{equation} \label{FasChi}
 {\mathcal F}^\perp(z/q) = \sum_{F \in {\mathcal M}} f^\perp(F) \Big( \frac{z}{q} \Big)^{\text{deg}(F)} = \sum_{n=0}^{\infty} \sigma^\perp(n) z^n   = \exp\Big(\sum_{k=1}^{n-1} \frac{\chi(k)}k z^k \Big),
 \end{equation}
and so one can rephrase the estimate in Theorem \ref{Halasz} as
\begin{equation} \label{NoH7}  
\left| \sigma(n) \right| \le \frac{\kappa^2}{n} 
\int_0^1 \Big(\exp\Big( \max_{|z|=\sqrt{t}} {\rm Re }\sum_{j=1}^{n-1} \frac{\chi(j)}{j} z^j \Big) \Big)\Big(\frac{1-t^{n-1}}{1-t}\Big) dt + \frac{\kappa}{n} . 
\end{equation}

For every fixed real number $\theta$, the function $f_{\theta}(M) = f(M) e(-\theta \ \text{deg}(M))$ is also multiplicative, with  $\chi_{\theta}(k) = \chi(k) e(-k\theta)$, and correspondingly $\sigma_{\theta}(n) = \sigma(n) e(-n\theta)$ (throughout we define $e(t):=e^{2\pi it}$).  This is analogous to the ``twist'' $f(n)n^{-i\theta}$ of a multiplicative function $f$ on the integers. Note that the inequalities in Theorem \ref{Halasz} and \eqref{NoH7}  remain unchanged if we replace $\chi$ by $\chi_\theta$ and $\sigma$ by 
$\sigma_\theta$.

The integral in Theorem \ref{Halasz} can be difficult to work with, and we now give a slightly weaker bound which is simpler to use.
By the maximum modulus principle we may bound $\max_{|z|=\sqrt{t}} |{\mathcal F}^{\perp}(z/q)|$ by $\max_{|z|=1} |{\mathcal F}^{\perp}(z/q)|$.  
Moreover, if $|z|\leq 1$ then 
\[
\log |{\mathcal F}^\perp(z/q)| \leq   \sum_{k=1}^{n-1} \frac{|\chi(k)|}k  \leq \kappa \sum_{k=1}^{n-1} \frac{1}k \leq \kappa \log (2n) ,
\]
so that $ |{\mathcal F}^\perp(z/q)|\leq (2n)^\kappa$.  In Section 3.2,  using these bounds appropriately in Theorem \ref{Halasz} 
we show the following corollary.

\begin{corollary} \label{HalCor} For all integers $n\ge 1$,  we have 
$$ 
\left| \sigma(n) \right|  \le 2\kappa   ( \kappa + 1 + M ) e^{-M} (2n)^{\kappa-1} ,
$$ 
where 
$$ 
 \max_{|z|=\frac{1}{q} } |{\mathcal F}^\perp(z)| =: e^{-M} (2n)^{\kappa}.
 $$ 
\end{corollary}

Since $M\geq 0$ we deduce from Corollary \ref{HalCor} the trivial bound $|\sigma(n)| \ll_{\kappa} n^{\kappa - 1}$  (see Remark 5 in Section \ref{examplesremarks} for a more precise estimate). Inserting {\em any} non-trivial lower bound for $M$ will yield a non-trivial bound for $|\sigma(n)|$, which as remarked earlier is the crucial feature of Hal\'{a}sz-type theorems.

In the integer situation, the bound corresponding to Corollary \ref{HalCor} is 
$$ 
\frac 1x \sum_{n\le x} f(n) \ll_\kappa  (1+M)e^{-M}   (\log x)^{\kappa -1} + \frac{(\log \log x)^{\kappa} }{\log x},
$$ 
where 
 $$
 \max_{|t| \le (\log x)^{\kappa}}  \Big| \frac{F^\perp(1+it)}{1+it}\Big| =: e^{-M} (\log x)^{\kappa}.
$$
This is usually stated with $F^\perp(1+it)$ replaced by $F(1+1/\log x+it)$; note that these two quantities are of comparable size (up to multiplicative constants).

\subsection{Lipschitz-type theorems for mean values of multiplicative functions over function fields} 

Mean values of  multiplicative functions in number fields vary slowly with $x$, provided one corrects by the ``rotation" of the form $n^{i\theta}$ that best approximates $f(n)$; that is, one replaces $f(n)$ by the twist 
$f_\theta(n):=f(n)n^{-i\theta}$ where $\theta$ maximizes $|F^\perp(1+i\theta)|$ with $|\theta|\leq\log x$. One can show that the mean value of $f$ up to $x$ equals $x^{i\theta}/(1+i\theta)$ times the mean value of $f_\theta$ plus a small error term. Let us restrict for simplicity to the case $\kappa = 1$. Building on Elliott's work \cite{Elliott}, in \cite{GSDecay} (and see also \cite{GHS}), we obtained the bound
\begin{equation} \label{NoH3} 
   \frac 1{x^{1+\phi}} \sum_{n\leq x^{1+\phi}} f_\theta(n)  - \frac 1{x} \sum_{n\leq x} f_\theta(n)  \ll 
\phi^{1-\frac 2\pi}\log \frac 2\phi ,
\end{equation}  
whenever $\frac{(\log\log x)^{2}}{\log x} \leq \phi \leq   1$. In \cite{GHS} we found examples showing the sharpness of \eqref{NoH3}, up to the factor of $\log 2/\phi$.  Thus the exponent $1-\frac 2\pi$ cannot be increased in general, and we say that $1-\frac 2\pi$ is the {\sl Lipschitz exponent} for mean values of multiplicative functions over the integers.
 
We now give analogous ``Lipschitz estimates" in the function field case, again restricting attention, for simplicity, to functions in ${\mathcal C}(1)$.

\begin{theorem} 
\label{Lipschitz} 
Let $f$ be in the class ${\mathcal C}(1)$, and let $\sigma$ and $\chi$ be defined as in \eqref{6} and \eqref{7}.  Let $n\ge 2$ be given, and let $f^{\perp} = f^{\perp,n}$, $\sigma^{\perp}$, $\chi^{\perp}$, and ${\mathcal F}^{\perp}(z)$ be defined as before. Select $\theta\in [0,1)$ for which $ |{\mathcal F}^\perp(e(-\theta)/q)|$ is maximized.
Then for any integer  $\ell$  with $1\le \ell \le n$, we have
$$ 
 \left|  \sigma_\theta(n+\ell) - \sigma_\theta(n) \right| \ll \Big( \frac{\ell}{n}\Big)^{1-\frac 2\pi} \log \frac{2n} {\ell} + \frac{(\log n)^{O(1)}}{n^{1-c_{m}} }  ,
$$
where $\sigma_\theta(k)=\sigma(k)e(-k\theta)$, and $m$ is the smallest odd integer that does not divide $\ell$, with $c_m:=1/(m\sin(\frac{\pi}{2m}))$.  
\end{theorem} 

Since $\sigma_\theta(n) =\sigma(n) e(-n\theta)$, Theorem \ref{Lipschitz} implies that 
$$ 
\Big| |\sigma(n+\ell)| - |\sigma(n)| \Big| \ll \Big( \frac{\ell}{n}\Big)^{1-\frac 2\pi} \log \frac{2n} {\ell} + \frac{(\log n)^{O(1)}}{n^{1-c_{m}} }.
$$ 

Note that $1 - c_m$ increases to $1 - 2/\pi$ as $m \rightarrow \infty$ through odd values. The first term of the upper bound in Theorem \ref{Lipschitz} corresponds to the upper bound in \eqref{NoH3}, but the second term indicates a new phenomenon, which has no parallel in the integer situation. 
In Section \ref{Construct}  we will construct examples, for each odd $m>1$, of $f$ for which
$| \sigma_\theta(n+\ell) - \sigma_\theta(n) |\gg 1/n^{1-c_m}$ for a positive proportion  of  integers $n$. 
For instance, if $m=3$  we take  $\chi(n)=1$ if $3$ divides $n$, and $\chi(n)=-1$ otherwise. Then one can show that
$\sigma(3n)=-\sigma(3n+1)\sim 1/(\Gamma(\frac 23)n^{1/3})$, whereas $\sigma(3n+2)=0$. Note here that $1-c_3=\frac 13$.
   
When $\ell$ is large, the first term in the Lipschitz bound dominates the second, and one can construct examples in which the exponent $1-2/\pi$ is attained.  Since this situation is in complete analogy with the situation for integers (see \cite{GHS}), we do not carry out this construction here.


\section{Some Examples and Remarks}\label{examplesremarks}

\noindent In this section we collect together some remarks on our class ${\mathcal C}(\kappa)$ and 
offer various motivating examples of functions in this class. The proofs of our theorems are deferred to the subsequent sections.
\medskip

\noindent {\bf Remark 1}.  The \emph{Dirichlet convolution} of $f_1 \in {\mathcal C}(\kappa_1)$ and $f_2 \in {\mathcal C}(\kappa_2)$ is given by   $(f_1* f_2)(F) := \sum_{AB=F} f_1(A) f_2(B)$, and lies in the class ${\mathcal C}(\kappa_1+\kappa_2)$.  The \emph{Rankin--Selberg convolution}, $f_1 \times f_2$,  is defined by setting $\Lambda_{f_1\times f_2}(M) =0$ unless $M$ is a prime power, in which case $\Lambda_{f_1\times f_2}(M) \Lambda(M) = \Lambda_{f_1}(M) \Lambda_{f_2}(M)$, so that $f_1 \times f_2$ is a 
multiplicative function in the class ${\mathcal C}(\kappa_1 \kappa_2)$.  The function $f_1 \times f_2$ matches 
the product $f_1f_2$ on squarefree $M$, but the two functions differ on prime powers $P^k$ with $k>1$ with the Rankin--Selberg convolution being the more natural choice.
\medskip

\noindent {\bf Remark 2}.  Let $f\in {\mathcal C}(\kappa)$, and let the averages $\sigma$ and $\chi$ be as in \eqref{6} and 
\eqref{7}.  Note that $\chi(0)=0$ and $|\chi(n)|\le \kappa$ for all $n$.  We now prove the convolution relation \eqref{2.1} satisfied by $\sigma$ and $\chi$.  First note that 
$$ 
f(F) \text{deg}(F) = \sum_{D | F} \Lambda_f(D) f(F/D), 
$$ 
which follows upon comparing the two sides of the relation $z{\mathcal F}^{\prime}(z) = 
(z{\mathcal F}^{\prime}/{\mathcal F}(z)){\mathcal F}(z)$.  Taking the average over $F \in {\mathcal M}_n$ gives
\begin{align*}
n \sigma(n) &= \frac{1}{q^n} \sum_{F \in {\mathcal M}_n} f(F) \text{deg}(F) = \frac{1}{q^n} \sum_{F \in {\mathcal M}_n}  \sum_{D|F } \Lambda_f(D) f(F/D)
\\
&= \sum_{k=1}^{n} \Big( \frac{1}{q^k} \sum_{D \in {\mathcal M}_k} \Lambda_f(D)\Big) \Big(\frac{1}{q^{n-k}} \sum_{M \in {\mathcal M}_{n-k}} f(M) \Big). 
\end{align*}
In other words we have the convolution identity 
\begin{equation*} 
n\sigma(n) = \sum_{k=1}^{n} \chi(k) \sigma(n-k). 
\end{equation*}
As discussed in the introduction, this is a simpler, discrete version of the integral equation
\[
u\sigma(u) = \int_0^u \chi(t) \sigma(u-t) dt 
\]
that occurs for number field mean values, which was first considered by Wirsing, and discussed further in \cite{GSSpectrum}. 
\medskip

\noindent{\bf Remark 3}.  The convolution identity \eqref{2.1} shows that $\sigma(n)$, the average of $f$ over elements of ${\mathcal M}_n$, depends only 
on the $\chi$-values, which are the average of $f$ taken over suitable prime powers, {\em but not on the individual $f(P^\ell)$}.  {\em Thus there is no 
loss in generality in assuming that $f(P^\ell)=\chi(k)$ whenever $\deg(P^\ell)=k$}. We will use this observation repeatedly in the sequel when discussing and constructing examples.  Further, this observation means that we can view Theorem \ref{Halasz} purely as a result in analysis: given information about the coefficients of the power series $\sum_{k=1}^{\infty} \chi(k)z^k/k$, it obtains information about the coefficients of the series $\sum_{n=0}^{\infty} \sigma(n) z^n = \exp(\sum_{k=1}^{\infty} \chi(k)z^k/k)$.  

 \medskip

\noindent {\bf Example 4}. In the introduction, we saw that if $f(P^k) =1$ for all prime powers $P^k$ then 
${\mathcal F(z)}=(1-qz)^{-1}$, and $\chi(k) = 1$ for all $k\ge 1$, and $\sigma(n)=1$ for all $n \ge 0$.  

Generalizing this, we consider the construction where $\chi(k)=\alpha$ for some fixed $\alpha\in {\Bbb C}$, and all $k\ge 1$.  We find (either by solving 
the recurrence \eqref{2.1}, or by noting that the corresponding generating function ${\mathcal F}(z)$ equals $(1-qz)^{-\alpha}$) that 
\begin{equation} 
\label{2.2}
\sigma(n) = \binom{\alpha+n-1}{n}  = \frac{\alpha (\alpha+1)\cdots (\alpha+n-1)}{n!}.
\end{equation} 
Therefore $\sigma(n) \sim n^{\alpha-1}/\Gamma(\alpha)$ for large $n$. This is analogous to the  Selberg-Delange theorem, which gives asymptotics for $\sum_{n\le x} d_z(n)$ where  $\zeta(s)^{z}=\sum_{n\geq 1}  d_z(n)/n^s$.  

When $\alpha =k \in {\Bbb N}$, the example above deals with the $k$-divisor function over ${\Bbb F}_q[x]$, 
and \eqref{2.2} gives the average number of ways of writing a polynomial $F$ of degree $n$ as the product $F_1\cdots F_k$.   The case $\alpha = -1$ deals with the analog of the M{\" o}bius function: $f(P)=-1$ for 
irreducibles $P$, and $f(P^\ell)=0$ for $\ell \ge 2$.  Finally, note that when $\alpha=-k$ is a negative integer then 
$\sigma(n) =0$ for all $n\ge k+1$.  
\medskip

\noindent{\bf Remark 5}.  If $f\in {\mathcal C}(\kappa)$, then by induction using \eqref{2.1} we see that for all $n\ge 0$, 
$$ 
|\sigma(n)| \le \binom{\kappa +n-1}{n}.
$$ 
\medskip

\noindent {\bf Example 6}. Smooth polynomials.  A $y$-\textsl{smooth integer} is a positive integer $n$, all of whose prime factors are $\leq y$. The indicator function of $y$-smooth integers is the completely multiplicative function $f$ for which $f(p)=1$ if $p\leq y$, and $f(p)=0$ otherwise.  It is known that for a wide range of $u\geq 1$ there are $\sim \rho(u)y^u$ $y$-smooth integers up to $y^u$, where $\rho(u) = e^{-(1+o(1))u\log u}$ is the Dickman function (which equals $1$ for $0\leq u\leq 1$, and is defined by $u\rho(u)=\int_0^1 \rho(u-t)dt$ for $u>1$).

Analogously an $m$-\textsl{smooth polynomial} is one all of whose irreducible factors have degree $\leq m$.
Consider the construction where $\chi(\ell) =1$ if $1\le \ell \le m$ and $\chi(\ell)=0$ for $\ell >m$. If we determine $\sigma(\cdot)$ using \eqref{2.1} then $N(n,m)$, the number of $m$-smooth monic 
polynomials of degree $n$, satisfies $N(n,m)\geq \sigma(n) q^n$.  Now $\sigma(n) =1=\rho(n/m)$ for $0\le n\le m$.  By an induction hypothesis and \eqref{2.1} one can then deduce that $\sigma(n)\ge \rho(n/m)$ for all $n$, as follows:   
$$ 
\sigma(n) \ge \frac{1}{n} \sum_{\ell = 1}^{m} \rho\Big(\frac{n-\ell}{m}\Big) \ge \frac 1n \sum_{\ell=1}^{m} \int_{\ell-1}^{\ell} \rho\Big(\frac{n-t}{m}\Big) dt = \frac 1n \int_0^m \rho\Big( \frac{n-t}{m} \Big) dt = \rho\Big( \frac{n}{m}\Big) ,
$$ 
the final equality holding since $u\rho(u) = \int_{u-1}^{u} \rho(v) dv$ for $u \geq 1$. In fact there is some room to spare in the lower bound above, and one can show that 
$$ 
\sigma(n) \ge \rho (n/m) \exp \Big( \frac{c}{m} \Big\lfloor \frac nm\Big\rfloor \Big), 
$$ 
for some positive constant $c$.   This implies that in the function field case, the Dickman function is {\sl not} a good approximation to $N(n,m)$ if $n \approx m^2$ (whereas it {\sl is} a good approximation in the corresponding range $u\approx \log y$ for the $y$-smooth integer counting problem).
A similar phenomenon occurs when we count ``smooth permutations"; that is, elements of $S_n$ composed of cycles of length at most $m$.  
\medskip

 \noindent {\bf Remark 7}.  In Example 4 we saw that if $\chi(\ell)=-k$ is a negative integer 
 for all $\ell \ge 1$ then $\sigma(n)$ equals zero for all $n\ge k+1$.  We now consider the converse situation: 
 if  $\sigma(n)=0$ for all $n\ge k+1$ (where $k\geq 0$), what can we conclude about $\chi(\ell)$?   Our assumption implies that
$$ 
{\mathcal F}(z) = \sum_{n=0}^{\infty} \sigma(n) (qz)^n = \sum_{n=0}^{k} \sigma(n) (qz)^{n} 
$$ 
is a polynomial of degree $k$.  Factoring ${\mathcal F}$ into its roots we obtain 
$$ 
{\mathcal F}(z) = \prod_{j=1}^{k} (1-z\alpha_j)
$$ 
for some complex numbers $\alpha_j$.   Since $\log {\mathcal F}(z) = \sum_{\ell=1}^{\infty} \frac{\chi(\ell)}{\ell} (qz)^{\ell}$ is holomorphic in the region $|z| <1/q$,  we see that $|\alpha_j| \le q$, or in other words all the zeros of ${\mathcal F}$ lie in $|z|\ge 1/q$.   Further we have 
$$ 
q^{\ell} \chi(\ell) = - \sum_{j=1}^{k} \alpha_j^{\ell}. 
$$ 
If $|\chi(\ell)|\le \kappa$ for all $\ell$, it follows that there can be at most $\kappa$ values of $\alpha_j$ with $|\alpha_j|=q$, and the rest are strictly smaller than $q$ in magnitude.  

If $f \in {\mathcal C}(\kappa)$ with $\kappa<1$ satisfies $\sigma(n)=0$ for $n \ge k+1$, then from the above we 
conclude that $\chi(\ell)$ must decrease exponentially for large $\ell$.  If $\kappa =1$, then either $|\alpha_j|<q$ for all $j$, 
in which case $\chi(\ell)$ once again decreases exponentially for large $\ell$, or $|\alpha_j| =q$ for some $j$ (say $j=1$).  
In the latter case, we may use Dirichlet's theorem to find $\ell$ such that all the $\alpha_j^{\ell}$ (for $1\le j\le k$) have 
argument in $(-\pi/8,\pi/8)$ say, and this forces $k=1$ (else one would find an $\ell$ with $|\chi(\ell)| >1$).  Thus in this 
case one must have ${\mathcal F}(z) = (1-qze(\theta))$ for some $\theta$;   in other words, {\sl the only possibility for $f$ is a twist of the M{\" o}bius function by some $\theta$}.   This is a simple analog of a striking converse theorem of Koukoulopoulos \cite{Koukou} for multiplicative functions over the integers.  It may be interesting to work out a precise analog of his result, which would involve imposing (given $f\in {\mathcal C}(1)$) 
the weaker restriction $|\sigma(n)| \ll n^{-2-\delta}$ for some $\delta>0$, and deriving a similar dichotomy for the behavior of $\chi(\ell)$.   

We proved that for any $\kappa>0$, at most $\lfloor \kappa\rfloor$ of the $\alpha_j$ can have size $q$. Here too, one would like to replace the condition that $\sigma(n)=0$ for large $n$, by a weaker condition like $\sigma(n) \ll n^{-A}$ for some $A=A(\kappa)$.   Koukoulopoulos and the third author have taken some first steps in this direction for multiplicative functions over the integers.  
\medskip

\noindent {\bf Remark 8.}   One of the main results in \cite{GSSpectrum} states that if $f: {\Bbb N} \to \{-1, 1\}$ is a completely multiplicative 
function, then for large $x$ one has $\sum_{n\le x} f(n) \ge (\delta_1 +o(1))x$, where 
$$ 
\delta_1 = 1 - 2\log (1+\sqrt{e}) + 4\int_1^{\sqrt{e}} \frac{\log t}{t+1} dt = -0.656999\ldots , 
$$ 
and the constant $\delta_1$ is optimal.   The exact parallel of this result is false in the function field setting: the function $f(F)= (-1)^{\text{deg}(F)}$ is completely multiplicative, and here $\sigma(n) = (-1)^n$.  It would be interesting to develop the right version of this result in the function field setting. Perhaps parity is the only substantial obstruction to such a result? 
\medskip
 
 \noindent {\bf Example 9}.  Let $e(\alpha_1)$, $\ldots$, $e(\alpha_k)$ be distinct points on the unit circle, and 
 let $a_1$, $\ldots$, $a_k$ be complex numbers, all bounded by $1$ say.  An interesting  class of examples 
 is given by setting (for $n\ge 1$) 
 $$ 
 \chi(n) = \sum_{j=1}^{k} a_j e(-n\alpha_j).
 $$ 
 Here, by \eqref{PowerSeries}, 
 $$
 \sum_{n= 0}^{\infty} \sigma(n) z^n = : {\mathcal F}(z/q) = \prod_{j=1}^{k} (1-ze(-\alpha_j))^{-a_j}. 
 $$ 
 By matching up the coefficients of the $(1-ze(-\alpha_j))^{-a_j}$,  we may construct a function 
 $$
 {\mathcal G}(z) = \sum_{j=1}^{k} \Big( C_0(j) (1-ze(-\alpha_j))^{-a_j} + C_1(j) (1-ze(-\alpha_j))^{1-a_j}\Big), 
 $$ 
 such that ${\mathcal F}(z/q) -{\mathcal G}(z)$, and its first derivative are bounded uniformly in $|z|<1$ (the bound may 
 depend on the $e(\alpha_j)$'s and  $a_j$, but remains uniform as $|z| \to 1$).   Here, for example,
 $$ 
C_0(j) = \prod_{\ell \neq j} (1-e(\alpha_j-\alpha_{\ell}))^{-a_{\ell}} 
$$ 
for each $j$, and the $C_1(j)$ are given by a similar but more complicated expression.  Since the first derivative of ${\mathcal F}(z/q) -{\mathcal G}(z)$ is bounded uniformly in $|z|<1$, it follows that the $n$-th coefficient of ${\mathcal F}(z/q)-{\mathcal G}(z)$ is $O(1/n)$.  
Thus we conclude that 
 \begin{align*}
 \sigma(n) &=\sum_{j=1}^{k} \Big( C_0(j) e(-n\alpha_j) \binom{a_j+n-1}{n} + C_1(j) e(-n\alpha_j) \binom{a_j-1+n-1}{n} \Big)+O\Big(\frac{1}{n}\Big)\\
 &= \sum_{j=1}^{k} C_0(j) e(-n\alpha_j) \frac{n^{a_j-1}}{\Gamma(a_j)} + O \Big(\frac{1}{n}\Big),
 \end{align*} 
 since each $|a_j|\leq 1$.

\section{Proofs of Theorem \ref{Halasz} and Corollary \ref{HalCor}} 

\noindent The key to our proof of Hal\'asz's Theorem over function fields, as well as over number fields, is an identity, given in Lemma \ref{lem1} below. As discussed in \cite{GHS}, the crucial feature of this identity is the presence of three generating functions in the integral on the right-hand side, which will allow us to bound that integral efficiently. There is a strong analogy with additive number theory, where ternary problems are accessible to harmonic analysis techniques (such as the circle method) but binary problems are usually not.
 
\begin{lemma} \label{lem1}  Let $f$ be any multiplicative function in the class ${\mathcal C}(\kappa)$, and let 
${\mathcal F}(z)$ be as in \eqref{3}.  Let $r$ be a positive real number with $r <1/q$.  Then 
$$ 
\sum_{M \in {\mathcal M}_n} f(M) = \frac{1}{n} \sum_{M \in {\mathcal M}_n} \Lambda_f(M) + \frac 1n \int_0^1 
\frac{1}{2\pi i} \int_{|z|=r} \Big(z\frac{{\mathcal F}^{\prime}}{{\mathcal F}}(z)\Big) 
\Big( tz \frac{{\mathcal F}^{\prime}}{{\mathcal F}} (tz) \Big) {\mathcal F}(tz) \frac{dz}{z^{n+1}} \frac{dt}{t}. 
$$ 
\end{lemma} 
\begin{proof} By Cauchy's formula we may write, for any $0<r<1/q$,
\begin{equation}
\label{3.1}
\sum_{M\in {\mathcal M}_n} f(M) = \frac{1}{n} \frac{1}{2\pi i} \int_{|z|=r} z {\mathcal F}^{\prime}(z) \frac{dz}{z^{n+1}}. 
\end{equation} 
Now we write 
\begin{align*}
z{\mathcal F}^{\prime}(z) &= \Big(z\frac{{\mathcal F}^{\prime}}{{\mathcal F}}(z) \Big) {\mathcal F}(z) = 
\Big(z\frac{{\mathcal F}^{\prime}}{{\mathcal F}}(z) \Big) \Big( 1+ \int_0^1 \frac{d}{dt} {\mathcal F}(tz) dt\Big)\\
&= 
\Big( z\frac{{\mathcal F}^{\prime}}{{\mathcal F}}(z) \Big)+ \Big( z\frac{{\mathcal F}^{\prime}}{{\mathcal F}}(z) \Big) 
\int_0^1 \Big(tz \frac{{\mathcal F}^{\prime}}{{\mathcal F}}(tz) \Big) {\mathcal F}(tz) \frac {dt}{t}, 
\end{align*}
and use this expression in \eqref{3.1}.   The first term above gives 
$$ 
\frac{1}{n} \frac{1}{2\pi i} \int_{|z|=r} \Big( z\frac{{\mathcal F}^{\prime}}{{\mathcal F}}(z) \Big) \frac{dz}{z^{n+1} } 
= \frac{1}{n} \sum_{M\in {\mathcal M}_n} \Lambda_f(M), 
$$ 
matching the first term in the right-hand side of the lemma.  The second term gives, upon interchanging the integrals over $z$ and $t$, the other term in the right-hand side of the lemma. 
\end{proof}

In \cite{GHS}, we use an analogous ``triple convolution'' identity  in our proof of  Hal\'asz's Theorem over number fields, but the key analytic technique there is Perron's formula rather than  Cauchy's formula, which leads to several additional complications.

 \subsection{Proof of Hal\'asz's Theorem in function fields}  
 
Theorem \ref{Halasz} clearly holds when $n=1$, and so we suppose below that $n\ge 2$.  
Since $\sigma(n)$ depends 
 only on the values of $f$ on prime powers with degree at most $n$, we are motivated to use the multiplicative 
 function $f^\perp (=f^{\perp ,n})$ as described in the introduction.  We recall that $\sigma^\perp(j) = \sigma(j)$ 
 for all $j\le n-1$, and that $\sigma^\perp(n) = \sigma(n) - \chi(n)/n$, and note that ${\mathcal F}^{\perp}(z)$ is an entire function for all $z\in {\Bbb C}$.

 Now use Lemma \ref{lem1} with ${\mathcal F}$ replaced by ${\mathcal F}^\perp$ there.  From our 
 observations above, we obtain (with ${\mathcal F}^{\perp \prime}(z)$ denoting the derivative of ${\mathcal F}^{\perp}(z)$)
 \begin{equation} 
 \label{3.2} 
 \sigma(n) -\frac{\chi(n)}{n} = \sigma^\perp(n) = \frac{q^{-n}}{n} \int_0^1 
 \frac{1}{2\pi i} \int_{|z|=r} \Big(z\frac{{\mathcal F}^{\perp \prime}}{{\mathcal F}^\perp} (z) \Big) \Big( tz \frac{{\mathcal F}^{\perp \prime}}{{\mathcal  F}^\perp}(tz)\Big) {\mathcal F}^\perp(tz) \frac{dz}{z^{n+1}} \frac{dt}{t}. 
 \end{equation} 
By \eqref{FasChi} we see $u{\mathcal F}^{\perp \prime}/{\mathcal F}^\perp(u) = \sum_{j=1}^{n-1} \chi(j) (qu)^j$ is a finite sum, so we may take 
the inner integral over $z$ in \eqref{3.2} to be over the circle with radius $1/(q\sqrt{t})$, and obtain 
\begin{equation} \label{3.3} 
\sigma(n)-\frac{\chi(n)}{n} = 
\frac{q^{-n}}{n} \int_0^1 \frac{1}{2\pi i} \int_{|z|=\frac{1}{q\sqrt{t}}} \Big( \sum_{j=1}^{n-1} \chi(j) (qz)^{j}\Big) \Big( \sum_{j=1}^{n-1} \chi(j)(qtz)^{j}\Big) {\mathcal F}^\perp(tz) \frac{dz}{z^{n+1}} \frac{dt}{t}. 
\end{equation}

Now consider the inner integral in \eqref{3.3}.  Using Cauchy--Schwarz we see that 
\begin{align}\label{3.4}
\Big| \frac{1}{2\pi i} \int_{|z|=\frac{1}{q\sqrt{t}}}  \Big( \sum_{j=1}^{n-1} \chi(j) (qz)^{j}\Big) \Big( \sum_{j=1}^{n-1} \chi(j)(qtz)^{j}\Big)& {\mathcal F}^\perp(tz) \frac{dz}{z^{n+1}} \Big| \nonumber\\
\le (q\sqrt{t})^n \Big( \max_{|z|= \frac{1}{q\sqrt{t}} } |{\mathcal F}^\perp(tz)| \Big) &
\Big( \frac{1}{2\pi } \int_{|z|= \frac{1}{q\sqrt{t}}} \Big| \sum_{j=1}^{n-1} \chi(j) (qz)^{j} \Big|^2 \frac{|dz|}{|z| } \Big)^{\frac 12} \nonumber\\
&\times \Big( \frac{1}{2\pi} \int_{|z|=\frac{1}{q\sqrt t} } \Big|\sum_{j=1}^{n-1} \chi(j) (qtz)^j \Big|^2 \frac{|dz|}{|z|}\Big)^{\frac 12}.   
\end{align} 
By Parseval, and since $|\chi(j)| \le \kappa$ for all $j$, we have 
\begin{equation} 
\label{3.5}
\frac{1}{2\pi } \int_{|z|=R} \Big| \sum_{j=1}^{n-1} \chi(j) (qz)^j \Big|^2 \frac{|dz|}{|z|} = \sum_{j=1}^{n-1} |\chi(j)|^2 (qR)^{2j} \le \kappa^2 \sum_{j=1}^{n-1} (qR)^{2j}. 
\end{equation}
Inserting this into \eqref{3.4}, we deduce that  \eqref{3.4} is 
$$ 
\le \kappa^2  (q\sqrt{t})^n \Big( \max_{|z|=\frac{\sqrt{t}}{q}} |{\mathcal F}^\perp(z)|\Big) 
\Big( \sum_{j=1}^{n-1} t^{-j}\Big)^{\frac 12} \Big( \sum_{j=1}^{n-1} t^j \Big)^{\frac 12}
= \kappa^2 q^n t \Big( \frac{1-t^{n-1}}{1-t}\Big)\Big( \max_{|z|=\frac{\sqrt{t}}{q}} |{\mathcal F}^\perp(z)|\Big). 
$$
Inserting this into \eqref{3.3} yields Theorem \ref{Halasz}.

  \subsection{Proof of Corollary \ref{HalCor}} 

As mentioned in the introduction, the maximum modulus principle gives
$$ 
\max_{|z|=\frac{\sqrt{t}}{q} } |{\mathcal F}^\perp(z)| \leq  \max_{|z|=\frac{1}{q} } |{\mathcal F}^\perp(z)| =: e^{-M} (2n)^{\kappa} 
 $$ 
for $0\leq t\leq 1$. Moreover, by definition when $|z| < 1/q$ we have
  \[
 \log |{\mathcal F}^\perp(z)|   = \text{Re} \Big( \sum_{k=1}^{n-1} \frac{\chi(k)}k (qz)^k \Big)  
 \leq \kappa \sum_{k=1}^{n-1} \frac{(q|z|)^k}k  \leq \kappa \sum_{k\geq 1}  \frac{(q|z|)^k}k= -\kappa \log (1-q|z|) .
 \] 
 Therefore
 \begin{equation} 
 \label{3.6} 
 \max_{|z|=\frac{\sqrt{t}}{q}} |{\mathcal F}^\perp(z)| \le \min \Big( e^{-M} (2n)^{\kappa}    , (1-\sqrt{t})^{-\kappa}\Big). 
 \end{equation}    
    
Taking $t= (1-u)^2$, and using \eqref{3.6}, we obtain  
\begin{align*}
& \kappa^{2} \int_0^1 \Big(  \max_{|z|=\frac{\sqrt{t}}{q}} |  {\mathcal F}^\perp(z) |\Big)  \Big(\frac{1-t^{n-1}}{1-t}\Big) dt \\
&\le \kappa^2 \int_0^1 \min \Big( e^{-M} (2n)^{\kappa}, u^{-\kappa}\Big )  \min 
\Big( (n-1),\frac{1}{u(2-u)} \Big) (2(1-u)) du\\
&\le \kappa^2 \int_0^1 \min \Big(e^{-M} (2n)^{\kappa}, u^{-\kappa}\Big )  \min 
\Big( 2n,\frac{1}{u} \Big)   du  \\
&\le \kappa^2 \Big( 
\int_0^{1/2n} 2n e^{-M} (2n)^{\kappa}  du
+ \int_{1/2n}^{e^{M/\kappa}/2n} e^{-M} (2n)^{\kappa} \frac{du}{u}
+ \int_{e^{M/\kappa}/2n} ^1  u^{-\kappa-1}    du
\Big) \\
&= \kappa^2 \Big(  e^{-M} (2n)^{\kappa}  +  e^{-M} (2n)^{\kappa} \frac{M}{\kappa} +  \frac{e^{-M} (2n)^{\kappa}-1}\kappa \Big)\\
&=  \kappa e^{-M} (2n)^{\kappa}  (  \kappa  +  1  +   M  ) - \kappa .
\end{align*}  
Substituting this bound into Theorem \ref{Halasz} yields Corollary \ref{HalCor}.

\section{Lipschitz estimates: A key proposition}  

\noindent Throughout this section, we restrict attention to $f\in {\mathcal C}(1)$, and prove an appropriate modification of Corollary \ref{HalCor} to bound the difference $\left|  \sigma_\theta(n+\ell) - \sigma_\theta(n) \right|$.

\begin{proposition} 
\label{Lipschitz2} 
Let $f$ be in the class ${\mathcal C}(1)$, and let $\sigma(n)$ and $\chi(n)$ be defined as in \eqref{6} and \eqref{7}.   Let $n\ge 2$. 
Define, for a given $\theta \in {\Bbb R}/{\Bbb Z}$, 
$$ 
L(n,\ell;\theta) = \max_{|z| = 1} \Big| (1 - z^{\ell} ) \ \exp\Big( \sum_{j=1}^{n-1} \chi_\theta(j) \frac{z^j}{j}\Big) \Big|.
$$ 
For any integer  $\ell\geq 1$, we have (recall $\sigma_\theta(k) =\sigma(k) e(-k\theta)$)
$$ 
\left|  \sigma_\theta(n+\ell) - \sigma_\theta(n) \right| \ll \frac{\ell}{n} + \frac{L(n,\ell;\theta)}{n} \Big(1+ \log \frac{2n}{L(n,\ell;\theta)}\Big).
$$
\end{proposition} 

We will apply this result for a suitable choice of $\theta$ in the next section, so as to deduce Theorem \ref{Lipschitz}.

\begin{proof} Let $\chi^\perp(k)$, $\sigma^\perp(k)$ and $\mathcal{F}^\perp(\cdot)$ be defined as in the introduction, so that $\chi^\perp(k) = \chi(k)$ when $k \leq n-1$, and $\chi^\perp(k)=0$ for larger $k$, and
\[
 {\mathcal F}^\perp(z/q) = \sum_{n=0}^{\infty} \sigma^\perp(n) z^n = \exp\Big(\sum_{k\le n-1} \frac{\chi(k)}{k} z^k \Big).
 \]

 If $k\le n-1$ we have $\sigma^\perp(k) = \sigma(k)$, and hence, using \eqref{2.1},  if $1 \le \ell \le n$ then
\[
(n+\ell) |\sigma(n+\ell) -\sigma^\perp(n+\ell)| \le  \sum_{k=1}^{n+\ell} |\chi(k)\sigma(n+\ell-k) - \chi^\perp(k) \sigma^\perp(n+\ell-k)| \le  2(2\ell+1),
\]
since the terms cancel out unless $k\le  \ell$ or $k\ge n$, and as $|\chi(\cdot)|$,  $|\chi^\perp(\cdot)|$, $|\sigma(\cdot)|$, $|\sigma^\perp(\cdot)|$ are all at most $1$.  Therefore we have
\begin{align}
n\left|  \sigma_\theta(n+\ell) - \sigma_\theta(n) \right| &= n\ |\sigma(n+\ell) e(-\ell \theta)-\sigma(n)| \notag \\ \notag
  &\leq |(n+\ell)\sigma(n+\ell)e(-\ell \theta) - n\sigma(n)| +\ell
\\ & \leq |(n+\ell)\sigma^\perp(n+\ell)e(-\ell \theta) - n\sigma^\perp(n)| +5\ell+ 4. \label{4.1} 
\end{align}
 
Following the above tidying up, we switch to our main analytic argument and apply Lemma \ref{lem1} to the first term on the right hand side, obtaining 
\begin{align} 
\label{4.2} 
(n+\ell)\sigma^\perp(n+\ell)e(-\ell \theta) &- n\sigma^\perp(n)  =\nonumber\\
&  \int_0^1 \frac{1}{2\pi i} \int_{|z|=r} \Big(z \frac{{\mathcal F}^{\perp \prime}}{{\mathcal F}^\perp}(z) 
\Big) \Big( tz \frac{{\mathcal F}^{\perp \prime}}{{\mathcal F}^\perp}(tz) \Big) {\mathcal F}^\perp(tz) \frac{e(-\ell\theta)(qz)^{-\ell}-1}{(qz)^{n}} \frac{dz}{z} \frac {dt}{t}. 
\end{align}  
As before we take the inner integral to be over the circle with radius $r=1/(q\sqrt{t})$.  Using that 
$u{\mathcal F}^{\perp \prime}/{\mathcal F}^\perp(u) = \sum_{j=1}^{n-1} \chi(j) (qu)^j$, and by Cauchy--Schwarz, 
we can bound the inner integral by 
\begin{align}
\label{4.3}
\max_{|z|=\frac{1}{q\sqrt{t}}} & |{\mathcal F}^\perp(tz) (e(-\ell \theta)(qz)^{-\ell}-1)| \times \nonumber \\
& t^{\frac n2}  \Big(\frac 1{2\pi} 
\int_{|z|=\frac{1}{q\sqrt{t}}} \Big|\sum_{j=1}^{n-1} \chi(j) (qz)^j\Big|^2 \frac{|dz|}{|z|}\Big)^{\frac 12} 
  \Big( \frac {1}{2\pi } \int_{|z|=\frac{1}{q\sqrt{t}}} \Big|\sum_{j=1}^{n-1} \chi(j) (qtz)^{j} \Big|^2 \frac{|dz|}{|z|}\Big)^{\frac 12}. 
\end{align}
Using the Parseval bound of \eqref{3.5} (with $\kappa=1$ there) and the display immediately following \eqref{3.5}, 
the second line of \eqref{4.3} is 
$$ 
\le t\Big(\frac{1-t^{n-1}}{1-t}\Big) \le t \ \min \left\{ n, \frac 1 {1-t} \right\} \ .
$$
To bound the maximum on the first line of \eqref{4.3}, we first let $w=qtz$ so that
\begin{eqnarray}
\max_{|z|=\frac{1}{q\sqrt{t}}}  |{\mathcal F}^\perp(tz) (e(-\ell \theta)(qz)^{-\ell}-1)| & = & \max_{|w|=\sqrt{t}}  |{\mathcal F}^\perp(w/q) (e(-\ell \theta)(w/t)^{-\ell}-1)| \nonumber \\
& = & \max_{|w|=\sqrt{t}}  |{\mathcal F}^\perp(w/q) (e(\ell \theta)w^{\ell}-1)| , \nonumber
\end{eqnarray}
where the final equality holds because $(w/t)^{-1}$ is the complex conjugate of $w$ when $|w|=\sqrt{t}$. By the maximum modulus principle, this is 
\[
\leq \max_{|w|=1}  |{\mathcal F}^\perp(w/q) (e(\ell \theta)w^{\ell}-1)| =
\max_{|w|=1} \ |e(-\ell \theta)w^{-\ell}-1| \cdot  \exp\Big(  \text{Re} \Big( \sum_{k\le n-1} \frac{\chi(k)}kw^k  \Big)\Big) .
\]
Writing $w=e(-\theta)z$, we see that this is $L(n,\ell;\theta)$.  
Since $|{\mathcal F}^\perp(z)| \le (1-|qz|)^{-1}$ for $|z|\leq 1/q$, we also have the alternative bound
\[
\max_{|z|=\frac{1}{q\sqrt{t}}} |{\mathcal F}^\perp(tz) (e(-\ell \theta)(qz)^{-\ell}-1)| \le 2 \max_{|z|=\frac{1}{q\sqrt{t}}} |{\mathcal F}^\perp(tz)| \le \frac{2}{1-\sqrt{t}} .
\]
Inserting these bounds into \eqref{4.3} and then \eqref{4.2} yields
\[
|(n+\ell)\sigma^\perp(n+\ell)e(-\ell \theta) - n\sigma^\perp(n)|\leq 
\int_0^1   \min \Big\{ n, \frac 1 {1-t} \Big\}  \min \Big \{ L(n,\ell;\theta), \frac{2}{1-\sqrt{t}}  \Big\} 
dt .
\]
Arguing as in the proof of Corollary \ref{HalCor}, we set $t=(1-u)^2$ so that the integral is
\begin{align*}
&\leq \int_0^1   \min \Big \{ 2n, \frac 1 {u} \Big\}  \min \Big\{ L, \frac{2}{u}  \Big\}   du 
\leq \int_0^{1/2n}  2nL  du
+\int_{1/2n}^{1/L} \frac L {u}  du
+\int_{1/L}^1 \frac 2 {u^2}  du\\
&= L    +   L \log(2n/L) +2L-2 
\end{align*}
where $L=L(n,\ell;\theta)$.  Inserting this into \eqref{4.1}, we obtain
\[
 n\ |\sigma_\theta(n+\ell)  -\sigma_\theta(n)|   \leq
 5\ell + 3L+  L\log(2n/L) +2,
\]
and the result follows.
 \end{proof}
 

\section{Lipschitz estimates: Proof of Theorem \ref{Lipschitz} }  
\noindent In the previous section we estimated $|\sigma_\theta(n+\ell)-\sigma_\theta(n)|$ in terms of the 
parameter $L(n,\ell;\theta)$ defined in Proposition \ref{Lipschitz2}.  
%
In the next lemma we show that if we choose $\theta$ such that {\rm Re}$( \sum_{j=1}^{n-1}  {\chi_\theta(j)}/{j})$ is maximized (which is the same as choosing it such that $ |{\mathcal F}^\perp(e(-\theta)/q)|$ is maximized), then
\[
L(n,\ell;\theta) \leq \max_{\alpha\in [0,1)} L^*(n,\ell;\alpha) \ \ \text{where} \ \ L^*(n,\ell;\alpha):=|1 - e(\ell \alpha)|  \ \exp\Big(   \sum_{k=1}^{n-1} \frac{|\cos(\pi k\alpha)|}{k}   \Big) .
\]
We then proceed to give accurate estimates, up to a constant, for each $L^*(n,\ell;\alpha)$ and, optimizing, deduce Theorem  \ref{Lipschitz}.

\subsection{Determining what is to be optimized}

 \begin{lemma} \label{lem2}  Select $\theta$ so as to maximize
{\rm Re}$( \sum_{j=1}^{n-1}  {\chi_\theta(j)}/{j})$. Then
\[
\Big| (1 - z^{\ell} ) \ \exp\Big( \sum_{j=1}^{n-1} \chi_\theta(j) \frac{z^j}{j}\Big) \Big|_{z=e(\alpha)}
  \leq  \ |1 - e(\ell \alpha)|  \ \exp\Big(   \sum_{k=1}^{n-1} \frac{|\cos(\pi k\alpha)|}{k}   \Big)
 \]
\end{lemma}
 \begin{proof}   If $z=e(\alpha)$ then, by the definition of $\theta$ as the maximizer,
 \begin{align*} 
 \text{\rm Re} \ \Big( \sum_{j=1}^{n-1}  \frac{\chi_\theta(j)}{j} z^j\Big) &\leq 
\frac 12\   \text{\rm Re} \ \Big( \sum_{j=1}^{n-1}  \frac{\chi_\theta(j)}{j} +  \sum_{j=1}^{n-1}  \frac{\chi_\theta(j)}{j} z^j\Big) \\
   &=    \text{\rm Re} \ \Big( \sum_{j=1}^{n-1}  \frac{\chi_\theta(j)e(j\alpha/2)}{j} \cos(\pi j\alpha)\Big)  
 \leq  \sum_{k=1}^{n-1} \frac{|\cos(\pi k\alpha)|}{k}. 
\end{align*}
This proves the lemma.  Note also that equality holds in the last step above when $\chi(k)=e(k(\theta-\alpha/2))\ \text{sign}(\cos(\pi k \alpha))$, and so the lemma is sharp in general.
 \end{proof}

\subsection{Upper bounds for a given $\alpha$}
 
\begin{lemma} \label{lem3}  Suppose $n \geq 2$ and $\alpha \in [0,1)$ are given.  Let $R:=\lceil \log n \rceil$, 
and select $m\leq 2R$ such that $|\alpha-b/m|\leq 1/(2mR)$ for some $(b,m)=1$. Then
 \[
L^*(n,\ell;\alpha)\asymp \|\ell \alpha\|  \  n^{\frac 2\pi }   \Big(\min\{ n,1/ \| m\alpha\|\}\Big)^{c_m-\frac 2\pi }  m^{O(1)}
\]
where  $\| t\|$ is the distance from $t$ to the nearest integer, and
\[
 c_m := \frac{1}{m} \sum_{a=0}^{m-1} | \cos(\pi a/m)| = 
 \begin{cases} 
 \frac {\text{\rm cosec}(\pi/2m)}{m} &\text{if} \ m \ \text{is odd}, \\
 \frac   {\cot(\pi/2m)}m &\text{if} \ m \ \text{is even}. 
 \end{cases} 
\]
An alternative expression for $c_m$ is 
\[
c_m    = \frac 2\pi \Big(  1   - 2
\sum_{\substack{r\geq 1 \\ m|r}} \frac{(-1)^r} {4r^2-1} \Big)  .
 \]
\end{lemma}


 \begin{proof}  The function $|\cos (\pi t)|$ is periodic with period $1$, and a little computation gives the Fourier expansion
\begin{align*}
|\cos(\pi t)| = \sum_{r\in {\Bbb Z}} \frac{(-1)^{r+1}}{2\pi (r^2-1/4)} e(rt) = \frac{2}{\pi} - \sum_{r=1}^{\infty} \frac{(-1)^r}{2\pi(r^2-1/4)} (2\cos (2\pi rt)).
\end{align*}
%
Therefore
\begin{equation} \label{Fourier}
\sum_{k=1}^{n-1} \frac{|\cos(\pi k\alpha)|}{k} = \frac 2\pi \Big(  \sum_{k=1}^{n-1} \frac{1}{k}   - 2
\sum_{r\geq 1} \frac{(-1)^r} {4r^2-1} \sum_{k=1}^{n-1} \frac{\cos(2\pi kr\alpha)}{k} 
\Big) .
\end{equation}
The first sum is $=\log n+O(1)$, and all subsequent sums over $k$ are  $\ll \log n$, so we may truncate the $r$-sum at $r\leq R$, 
with an error of $O(1)$.

Let $S(x):=\sum_{k\leq x} \cos(2\pi kr\alpha)$, which is easily seen to be $ \ll 1/\| r\alpha\|$. 
By partial summation, we deduce that if $K\gg 1/\| r\alpha\|$ then
\[ 
\sum_{k\geq K} \frac{\cos(2\pi kr\alpha)}{k} =\int_K^\infty \frac{dS(t)}t = \frac{S(K)}K + \int_K^\infty \frac{S(t)}{t^2} dt\ll \frac 1{K \| r\alpha\|} \ll 1 .
\]
Therefore
\begin{align*}
 \sum_{k=1}^{n-1} \frac{\cos(2\pi kr\alpha)}{k} = \sum_{k=1}^{\min\{  1/\|r\alpha\| , n\}} \frac{1}{k}  + O(1) 
= \log (\min\{  1/\|r\alpha\| , n\} ) +O(1) . 
 \end{align*}
 
For each $r\leq R$  we have $|r\alpha-rb/m|\leq r/2mR\leq 1/2m$. 
Therefore if $m\nmid r$ then $\| r\alpha\|\geq 1/2m$, and so $\log (\min\{  1/\|r\alpha\| , n\} ) = \log (1/\|rb/m\| )+O(1)$. Since this is $\ll \log m$ we see in particular that the terms in the $r$-sum, for which $R\geq r>\log m$ and $m\nmid r$, contribute $O(1)$.  Therefore the contribution of the terms $r\le R$ with $m\nmid r$ to \eqref{Fourier} is 
\begin{equation} 
\label{Fourier2} 
-\frac{4}{\pi} \sum_{1\le r\le \log m} \frac{(-1)^r}{4r^2-1} \log \frac{1}{\| rb/m\|} + O(1) = O(1+\log m). 
\end{equation}

Note also that if $r\leq R$ and $m|r$ then $\| r\alpha\|=(r/m)\ \| m\alpha\|$, and so the contribution of these terms to \eqref{Fourier} is 
\[ 
-\frac{4}{\pi} \sum_{\substack{1\le r\le R\\ m|r}} \frac{(-1)^r}{4r^2-1} \log \min\Big(n, \frac{1}{(r/m) \| m\alpha\|}\Big) +O(1) = 
\Big(c_m -\frac 2\pi\Big) \log \min \Big(n,\frac{1}{\| m\alpha\|}\Big) + O(1).
\]
Using these observations in \eqref{Fourier}, we deduce the lemma.  
\end{proof}

\noindent {\bf Remark 10}.  From \eqref{Fourier2}, we see that the $m^{O(1)}$ term in Lemma \ref{lem3}  can be replaced by 
the more precise expression 
\[
\exp\Big ( - \frac 4\pi    \sum_{1\leq r\leq \log m} \frac{(-1)^r} {4r^2-1} \log (m/|(rb)_m| ) \Big) ,
\]
where $(t)_m$ is the least residue of $t\pmod m$, in absolute value. This can be shown to be  
$\ll m^{1-\frac 2\pi}$ (which is attained when $b=1$) and 
$\gg m^{\frac 12-\frac 2\pi}$ (which is attained when $2b\equiv 1 \pmod m$).

 \begin{corollary} \label{cor4}  Let $m_0$ be the smallest odd integer that does not divide $\ell$. 
Then 
\[
\max_{\alpha\in [0,1)} L^*(n,\ell;\alpha) \asymp \max\{ n^{c_{m_0}}  m_0^{O(1)},  n^{\frac 2\pi }  \ell^{1-\frac 2\pi } \} .
\]
\end{corollary}

\begin{proof}
Observe first, for use later, that $c_1 = 1$ and $c_2 = 1/2$, and that the $c_m$ tend upwards to $2/\pi$ as $m$ varies over even values, and they tend downwards to $2/\pi$ as $m$ varies over odd values. Note also that if $m$ is odd we have $\frac{2}{\pi} + \frac{C}{m^{2}} > c_{m} > \frac{2}{\pi} + \frac{c}{m^{2}}$, for certain absolute constants $C, c > 0$.

Let $0 \leq \alpha < 1$, and correspondingly choose $1 \leq m \leq 2 \lceil \log n \rceil$ as in Lemma \ref{lem3}. First we shall establish the upper bound implicit in Corollary \ref{cor4}.

In the case where $m$ is even we have $c_m - \frac{2}{\pi} < 0$, so Lemma \ref{lem3} directly implies that
$$ L^*(n,\ell;\alpha) \ll \|\ell \alpha\|  \  n^{\frac{2}{\pi}} m^{O(1)} \leq n^{\frac{2}{\pi}} m^{O(1)} . $$
Since $m_0$ is odd we have $c_{m_0} > \frac{2}{\pi} + \frac{c}{m_0^{2}}$, and since also $m, m_0 \ll \log n$ we conclude that $ L^*(n,\ell;\alpha) \ll n^{c_{m_0}}  m_0^{O(1)}$ if $m$ is even, as required.

In the case where $m$ is odd we have $c_m - \frac{2}{\pi} > 0$, and we need to divide into some sub-cases. Firstly, if $m \geq m_0$ then $c_m \leq c_{m_0}$, and so Lemma \ref{lem3} implies similarly as before that
$$ L^*(n,\ell;\alpha) \ll \|\ell \alpha\|  \  n^{c_m } m^{O(1)} \leq n^{c_m } m^{O(1)} \ll n^{c_{m_0}}  m_0^{O(1)} . $$

If $m$ is odd and $m < m_0$ then we must have $m \mid \ell$, by definition of $m_0$. Now there are two further sub-cases. Firstly, if $\|m \alpha\| \leq 1/\ell$ then we use the bound $\|\ell \alpha\| \leq (\ell/m) \|m \alpha\| \leq \ell \|m \alpha\|$ in conjunction with Lemma \ref{lem3}, obtaining
$$ L^*(n,\ell;\alpha) \ll \|\ell \alpha\| \  n^{\frac 2\pi } \| m\alpha\|^{\frac 2\pi - c_m }  m^{O(1)} \leq \ell n^{\frac 2\pi } \| m\alpha\|^{1 + \frac 2\pi - c_m }  m^{O(1)} \leq n^{\frac{2}{\pi}} \ell^{c_m - \frac{2}{\pi}} m^{O(1)} ,  $$
since $\|m \alpha\| \leq 1/\ell$. Then since $m < m_0 \ll \log(2\ell)$, and $c_m \leq c_3 = 2/3$ unless $m=1$, this bound is $\ll n^{\frac 2\pi }  \ell^{1-\frac 2\pi }$, which is acceptable.

The remaining sub-case is where $m < m_0$ is odd but $\|m \alpha\| > 1/\ell$, in which case Lemma \ref{lem3} gives
$$ L^*(n,\ell;\alpha) \ll \|\ell \alpha\| \  n^{\frac 2\pi } \| m\alpha\|^{\frac 2\pi - c_m }  m^{O(1)} \leq n^{\frac 2\pi } \| m\alpha\|^{\frac 2\pi - c_m }  m^{O(1)} \leq n^{\frac{2}{\pi}} \ell^{c_m - \frac{2}{\pi}} m^{O(1)} ,  $$
which is acceptable as in the previous sub-case.

We complete the proof by establishing the lower bound implicit in the corollary: namely,  $ \max_{\alpha\in [0,1)} L^*(n,\ell;\alpha) \gg \max\{ n^{c_{m_0}}  m_0^{O(1)},  n^{\frac 2\pi }  \ell^{1-\frac 2\pi } \}$.   We show that the first term in the max is attained when $\alpha=b/m_0$ with $(b,m_0)=1$, and the second term in the max is attained when $\alpha$ is roughly of size $1/\ell$.  Indeed, if $\alpha=b/m_0$ for some $(b,m_0)=1$ then Lemma \ref{lem3} (with $m=m_0$) gives
$$
L^*(n,\ell;\alpha)\asymp \Big\|\frac{\ell b}{m_0} \Big\|  \  n^{c_{m_0}} m_0^{O(1)} \asymp n^{c_{m_0}}m_0^{O(1)} , 
$$
since $m_0$ does not divide $\ell$.

This suffices, unless $\ell$ is at least a large multiple of $\log n$ (otherwise $n^{\frac 2\pi }  \ell^{1-\frac 2\pi } < n^{c_{m_0}}  m_0^{O(1)}$, since $c_{m_0} > \frac{2}{\pi} + \frac{c}{m_0^{2}} > \frac{2}{\pi} + \frac{c}{\log^{2}(2\ell)}$). If $\ell$ is at least a large multiple of $\log n$, then for any $1/(4\ell) \leq \| \alpha \| \leq 1/(2\ell)$ we can take $m=1$ in Lemma \ref{lem3}, and obtain
$$ L^*(n,\ell;\alpha)\asymp \|\ell \alpha\|  \  n^{\frac 2\pi } \ \| \alpha\|^{\frac 2\pi - c_1 } \asymp \ell n^{\frac 2\pi } \ \| \alpha\|^{1 + \frac 2\pi - c_1 } \asymp n^{\frac 2\pi }  \ell^{1-\frac 2\pi } , $$
since $c_1 = 1$.
\end{proof}

\begin{proof} [Proof of Theorem \ref{Lipschitz}] As noted earlier, if $\theta$ in Proposition \ref{Lipschitz2} is chosen as in 
Lemma \ref{lem2}, then $L(n,\ell;\theta) \le \max_{\alpha\in [0,1)} L^*(n,\ell;\alpha)$.  Now using the bound of Corollary \ref{cor4}
in Proposition \ref{Lipschitz2},   we obtain Theorem \ref{Lipschitz}. We record that, using Remark 10, one can replace $(\log n)^{O(1)}$ in the statement of Theorem \ref{Lipschitz} by $(\log n)^{2-\frac 2\pi}$.
\end{proof}

\section{Final examples}  \label{Construct}

\noindent We now give examples which establish that the upper bound in Theorem  \ref{Lipschitz} is attained for each fixed odd $m$. These are based on the discussion of Example 9, with a suitable choice of the points $\alpha_j$ there.
 
Suppose that for integers $k\ge 1$, the function $\chi(k)$ is periodic $\pmod m$.   Write
$$ 
{\hat \chi}(j) = \frac{1}{m} \sum_{k=1}^{m} \chi(k) e\Big( -\frac{jk}{m}\Big), 
$$ 
so that for all $n\ge 1$ we have 
$$
 \chi(n) = \sum_{j=1}^{m} {\hat \chi}(j) e\Big(\frac{jn}{m}\Big).
 $$
This class of examples satisfies the hypothesis of Example 9, and so the corresponding solution $\sigma(n)$ is given by the generating function 
\[
\sum_{n\geq 0} \sigma(n) z^n = \prod_{j=1}^{m}  (1-ze(j/m))^{- {\hat \chi}(j) }.
\]
Arguing as in Example 9, we may find asymptotics for $\sigma(n)$.

We now consider the special case where $m>1$ is odd, and $\chi(k) = \text{sign}(\cos(2\pi k/m))$ for all $k\ge 1$.  
Thus $\chi(k)=1$ when $\| k/m\| <\frac 14$ and $\chi(k) =-1$ otherwise. Recall from Remark 3 above that we are free to construct examples simply by specifying the behaviour of $\chi(k)$. Since $\chi(k)=\chi(m-k)$ for all $1\le k\le m-1$ we see that 
$$ 
{\hat \chi}(j) = \frac{\chi(m)}{m} + \frac{1}{m} \sum_{k=1}^{m-1}  \cos(2\pi jk/m) \chi(k) 
= \frac{1}{m} \sum_{k=1}^{m} \cos(2\pi jk/m) \chi(k). 
$$
In particular
$$ 
{\hat \chi}(1)  = {\hat \chi}(m-1) = \frac{1}{m} \sum_{k=1}^{m} \left|\cos(2\pi k/m) \right| = c_m, 
$$
whereas ${\hat \chi}(j)$ is a real number smaller than $c_m$ for all other values of $j$ 
(because the values $\chi(k)$ no longer perfectly ``resonate'' with the coefficients $\cos(2\pi jk/m)$). Therefore, by Example 9, we have
$$  \sigma(n) =\Big( C_0(1) e(n/m) +C_0(m-1) e(-n/m)+o(1)\Big)  \frac{n^{c_m-1}}{\Gamma(c_m)}  = C \Big( \cos \Big( \frac{2\pi n}{m} + \beta \Big)  +o(1) \Big) \frac{n^{c_m-1}}{\Gamma(c_m)} $$
for a suitable non-zero real number $C$, and some $\beta$.

Now Theorem \ref{Lipschitz} asserts that for any $1 \leq \ell \leq n$ such that $m$ is the smallest odd integer not dividing $\ell$, we will have
$$ 
\left|  \sigma_\theta(n+\ell) - \sigma_\theta(n) \right| = \left|  \sigma(n+\ell) e(-\ell\theta) - \sigma(n) \right| \ll \Big( \frac{\ell}{n}\Big)^{1-\frac 2\pi} \log \frac{2n} {\ell} + \frac{(\log n)^{O(1)}}{n^{1-c_{m}} }, 
$$
for a suitable $\theta\in [0,1)$. 
We now show that we will have $\left|  \sigma(n+\ell) e(-\ell\theta) - \sigma(n) \right| \gg 1/n^{1-c_m}$ for some (in fact a positive proportion of) $n$ values, for any such fixed $\ell$, so that the second term in the upper bound is sharp up to logarithmic factors. Indeed, for all $\theta$ we have 
\begin{align*}
\left|  \sigma(n+\ell) e(-\ell\theta) - \sigma(n) \right| &\geq \Big||\sigma(n+\ell)| - |\sigma(n)| \Big| \\
&=\Big|C \frac{n^{c_m-1}}{\Gamma(c_m)} \Big\{ \Big|\cos \Big( \frac{2\pi (n+\ell)}{m} + \beta \Big)\Big| -
\Big |\cos \Big( \frac{2\pi n}{m} + \beta \Big)\Big| + o(1) \Big\}\Big| . 
\end{align*}
This will  be $\gg_{m} 1/n^{1-c_m}$ unless $( \frac{2\pi n}{m} + \beta ) \equiv \pm ( \frac{2\pi (n+\ell)}{m} + \beta ) \pmod {\pi}$, and since $m$ is odd and $m \nmid \ell$ that can only happen if $( \frac{2\pi n}{m} + \beta ) \equiv - ( \frac{2\pi (n+\ell)}{m} + \beta ) \pmod {\pi}$, which (since $m$ is odd) can only happen for those $n$ in one particular residue class $\pmod m$. This justifies the remarks made at the end of the introduction concerning the optimality of the exponents appearing in Theorem \ref{Lipschitz}.


To give a more explicit example, when $m=3$ the analysis of Example 9 gives, for $\omega=e(\frac 13)$,
\[ 
\sum_{n\geq 0} \sigma(n) z^n = \frac{(1-z)^{1/3}}{ ((1-\omega z)(1-\overline{\omega} z))^{2/3}} = 
\frac{1-z}{(1-z^3)^{2/3}},
\]
and indeed that $\sigma(3n)=-\sigma(3n+1)=\Gamma(n+\frac 23)/n!\Gamma(\frac 23)\sim n^{-1/3}/\Gamma(\frac 23)$, and $\sigma(3n+2)=0$.
\medskip


\bibliographystyle{plain}
\bibliography{HalaszRefs}{}
 
\end{document}